\documentclass[10pt]{article}

\usepackage{amssymb}
\usepackage{amsmath}
\usepackage{theorem}
\usepackage{epsfig}
\usepackage{verbatim}
\usepackage{graphicx}
\usepackage{subfigure}
\usepackage{cancel}
\usepackage{epstopdf}

\usepackage{color}

\newtheorem{theorem}{Theorem}[section]
\newtheorem{lemma}[theorem]{Lemma}

\newtheorem{conjecture}[theorem]{Conjecture}

\newtheorem{rem}[theorem]{Remark}

\newcommand{\RR}{\mathbb{R}}

\newcommand{\da}{\partial_{\alpha}}

\newcommand{\al}{\alpha}
\newcommand{\eps}{\epsilon}

\newenvironment{proof}{\begin{trivlist} \item[] {\em Proof:}}{\hfill $\Box$
                       \end{trivlist}}

\makeatletter
\renewcommand*\l@section{\@dottedtocline{1}{0em}{1.5em}}
\renewcommand*\l@subsection{\@dottedtocline{2}{1.5em}{2.3em}}
\renewcommand*\l@subsubsection{\@dottedtocline{3}{3.8em}{3.7em}}
\makeatother

\numberwithin{equation}{section}

\allowdisplaybreaks[4]

\begin{document}

\title{A note on stability shifting for the Muskat problem}

\author{Diego C\'ordoba, Javier G\'omez-Serrano and Andrej Zlato\v{s}}

\maketitle

\begin{abstract}

In this note, we show that there exist solutions of the Muskat problem that shift stability regimes: they start unstable, then become stable, and finally return to the unstable regime. We also exhibit numerical evidence of solutions with medium-sized $L^{\infty}$ norm of the derivative of the initial condition that develop a turning singularity.


\end{abstract}

%\tableofcontents

\section{Introduction}

In this paper, we study two incompressible fluids with the same viscosity but different densities, $\rho^{+}$ and $\rho^{-}$, evolving in a two dimensional porous medium with constant permeability $\kappa$. The velocity $v$ is determined by Darcy's law
\begin{equation}\label{IIIdarcy}
\mu\frac{v}{\kappa}=-\nabla p-g\left(\begin{array}{cc}0\\ \rho\end{array}\right),
\end{equation}
where $p$ is the pressure, $\mu>0$ viscosity, and $g > 0$ gravitational acceleration. In addition, $v$ is incompressible:
\begin{equation}\label{IIIincom}
\nabla\cdot v=0.
\end{equation}
By rescaling properly, we can assume $\mu=g=1$. The fluids also satisfy the conservation of mass equation
\begin{equation}\label{IIIconser}
\partial_t\rho+v\cdot\nabla\rho=0.
\end{equation}

This is known as the Muskat problem \cite{Muskat:porous-media}. We denote by $\Omega^{+}$ the region occupied by the fluid with density $\rho^{+}$ and by $\Omega^{-}$ the region occupied by the fluid with density $\rho^{-} \neq \rho^{+}$. The point $(0,\infty)$ belongs to $\Omega^{+}$, whereas the point $(0,-\infty)$ belongs to $\Omega^{-}$. All quantities with superindex $\pm$ will refer to $\Omega^{\pm}$ respectively. The interface between both fluids at any time $t$ is a curve $z(\alpha,t)$. 
We will work in the setting of flat at infinity interfaces, although the results can be extended to the horizontally periodic case.

A quantity that will play a major role in this paper is the Rayleigh-Taylor condition, which is defined as
\begin{align*}
RT(\alpha,t)=-\left[ \nabla p^{-}(z(\alpha,t))-\nabla p^{+}(z(\alpha,t)) \right]\cdot\partial_\alpha^\bot z(\alpha,t),
\end{align*}
where we use the convention $(u,v)^{\perp} = (-v,u)$. If $RT(\alpha,t)>0$ for all $\alpha\in\mathbb{R}$, we will say that the curve is in the Rayleigh-Taylor stable regime at time $t$, and if $RT(\alpha,t) \leq 0$ for some $\alpha\in\mathbb{R}$, we will say that the curve is in the Rayleigh-Taylor unstable regime. 

One can rewrite the system \eqref{IIIdarcy}--\eqref{IIIconser} in terms of the curve $z(\al,t)$, obtaining
\begin{align}\label{muskatinterface}
\partial_{t} z(\al,t) = \frac{\rho^{-} - \rho^{+}}{2\pi} \int_\mathbb{R} \frac{z_1(\al,t) - z_1(\beta,t)}{|z(\al,t) - z(\beta,t)|^{2}}(\partial_{\al}z(\al,t) - \partial_{\beta}z(\beta,t)) d\beta.
\end{align}
A simple calculation of the Rayleigh-Taylor condition in terms of $z(\al,t)$ yields
\begin{align*}
 RT(\al,t) = g(\rho^{-} - \rho^{+})\partial_{\al} z_{1}(\al,t).
\end{align*}
When the interface is a graph, parametrized as $z(\al,t)=(\alpha,f(\alpha,t))$, equation \eqref{muskatinterface} becomes
\begin{align}\label{muskatinterfacegraph}
 \partial_t f(x,t) = \frac{\rho^{-} - \rho^{+}}{2\pi} \int_\mathbb{R} \frac{x-y}{(x-y)^2 + (f(x,t)-f(y,t))^2}(\partial_{x}f(x,t) - \partial_{y}f(y,t)) d y
\end{align}
and the Rayleigh-Taylor condition simplifies to
\begin{align*}
RT(\al,t) =g(\rho^{-}-\rho^{+}).
\end{align*}
The curve is now in the RT stable regime whenever $\rho^{+}<\rho^{-}$, i.e., the denser fluid is at the bottom.

The Muskat problem  has been studied in many works. A proof of local existence of classical solutions in the Rayleigh-Taylor stable regime and ill-posedness in the unstable regime appears in \cite{Cordoba-Gancedo:contour-dynamics-3d-porous-medium}. %\todo{Rafa: seguro? ojo porque he cambiado la frase}%
A maximum principle for $\|f(\cdot,t)\|_{L^\infty}$ can be found in \cite{Cordoba-Gancedo:maximum-principle-muskat}. Moreover, the authors showed in \cite{Cordoba-Gancedo:maximum-principle-muskat} that if $\|\partial_{x} f_0\|_{L^\infty}<1$, then $\|\partial_{x} f(\cdot,t)\|_{L^\infty}<\|\partial_{x} f_0\|_{L^\infty}$ for all $t>0$. Further work has shown existence of finite time turning (i.e., the curve ceases to be a graph in finite time and the Rayleigh-Taylor condition changes sign) \cite{Castro-Cordoba-Fefferman-Gancedo-LopezFernandez:rayleigh-taylor-breakdown}. The precise result is the following:

\begin{theorem}\label{oneone}
There exists a nonempty open set of initial data in $H^4$ with
Rayleigh-Taylor strictly positive (i.e., $RT(\al,0)>0$ for all $\al\in\mathbb{R}$) such that the solutions of \eqref{muskatinterface} have $RT(\alpha,t)<0$ for some finite time $t>0$ and all $\alpha$ in some nonempty open interval.
\end{theorem}

After this shift of regime, the curve may lose regularity. This was proved in \cite{Castro-Cordoba-Fefferman-Gancedo:breakdown-muskat}. More general models which take into account finite depth or non-constant permeability that also exhibit turning were studied in \cite{Berselli-Cordoba-GraneroBelinchon:local-solvability-singularities-muskat,GomezSerrano-GraneroBelinchon:turning-muskat-computer-assisted}, where the estimates in the latter one were carried out by rigorous computer-assisted techniques, as opposed to the traditional pencil and paper ones from the former. All these results are local in time, and the techniques employed to prove them rely on a local analysis of the equations, therefore it is not possible to conclude global properties of the solutions from  them. Concerning global existence, the first proof for small initial data was carried out in \cite{Siegel-Caflisch-Howison:global-existence-muskat} in the case where the fluids have different viscosities and the same densities (see also \cite{Cordoba-Gancedo:contour-dynamics-3d-porous-medium} for the setting of this paper: different densities and the same viscosities). 
Global existence for medium-sized initial data was established in \cite{Constantin-Cordoba-Gancedo-Strain:global-existence-muskat,Constantin-Cordoba-Gancedo-RodriguezPiazza-Strain:muskat-global-2d-3d}. In the case where surface tension is taken into account, global existence was shown in \cite{Escher-Matioc:parabolicity-muskat,Friedman-Tao:nonlinear-stability-muskat-capillary-pressure}. Global existence for the confined case was treated in \cite{GraneroBelinchon:global-existence-confined-muskat}. Recently, in \cite{Cheng-GraneroBelinchon-Shkoller:well-posedness-h2-muskat}, a new framework was used to study global existence. The following theorem was proved in \cite{Constantin-Cordoba-Gancedo-Strain:global-existence-muskat}:

\begin{theorem}\label{weakSOLthm}
Suppose that $\|f_0\|_{L^\infty}<\infty$ and $\|\partial_x f_0\|_{L^\infty}<1$.Then there exists a global in time weak solution of \eqref{muskatinterfacegraph} that satisfies 
\begin{align*}
f\in C([0,T]\times\RR)\cap L^\infty([0,T];W^{1,\infty}(\RR)).
\end{align*}
In particular, $f$ is Lipschitz continuous.
\end{theorem}

The question of whether the constant 1 (which is independent of the physical parameters of the system) in the $L^{\infty}$ bound of the derivative of the initial data is sharp or not is still open. The proof of Theorem \ref{oneone} constructs initial conditions that have parameters that need to be taken big enough to ensure turning. Taken both these two facts into account, our initial motivation was to quantify and bridge the gap, finding or approximating the threshold $C$ for the $L^{\infty}$ norm of $\partial_{x} f_0$ such  that initial data for which $\|\partial_{x} f_0\|_{L^{\infty}} < C$ exist globally, and  for each $\eps > 0$ there exist solutions with $\|\partial_{x} f_0\|_{L^{\infty}} = C+\eps$ which develop turning singularities.

This note is organized as follows: in Section \ref{SectionNumericalResults} we describe the numerical experiments that became the motivation for our main theorem, which we prove in Section \ref{SectionTheorem}. Finally, in Section \ref{SectionDiscussion} we discuss open problems and future work.

\section{Numerical results}
\label{SectionNumericalResults}

% Decir motivacion, ecuacion del calor
In this section we describe numerical experiments which led to  our main result, Theorem \ref{theoremshifting} below.

Our motivation was to find initial data, as small as possible, which develop turning in finite time. For simplicity, we performed simulations in the horizontally periodic scenario. The evolution equation for the interface reads in that case:
% Escribir ecuaciones en el dominio periodico
\begin{align}
\label{muskatperiodico}
\partial_t z(\al,t) = \frac{\rho^{-} - \rho^{+}}{4\pi} \int_{\mathbb{T}} (z(\al,t) - z(\beta,t))\frac{\sin(z_{1}(\al,t) - z_{1}(\beta,t))}{\cosh(z_{2}(\al,t) - z_{2}(\beta,t)) - \cos(z_{1}(\al,t) - z_{1}(\beta,t))} d\beta.
\end{align}
Picking different initial conditions and evolving them until they either develop turning or flatten out seems like looking for a needle in a haystack. Instead, we took data which we were sure that would turn (the interface has a vertical tangent at a point and the velocity is pointing in the right direction) and run the equation \emph{backwards} in time  for a short time to find our desired initial condition.  

At the linear level and backwards in time, equation \eqref{muskatperiodico} behaves like a backwards heat equation, which is ill posed if the lighter fluid is on top of the denser one. To overcome this difficulty, we use the following heuristic: we perform very small steps backwards in time and at every step, we smooth the function by eliminating all frequencies whose components are below a given threshold. 
The reason behind this heuristic is that the family of solutions which turn over is an open set, and therefore small perturbations (the regularized versions) of the solution should remain in it.
% the set while retaining the low values of the derivative. 
We remark that the backwards evolution is not done with the purpose of finding a numerical solution of the equation, but only to gain intuition about what initial condition to choose. Once we find a suitable candidate, we check its validity by running the equation forward with the candidate. 

The smoothing threshold was set to $10^{-6}$. The time integration was done using a Runge-Kutta 45 scheme. The derivatives were calculated using a spectral method with a cutoff filter given by
\begin{align*}
 \rho(k) = \exp\left(-10\left(\frac{2|k|}{N}\right)^{25}\right), \quad |k| \leq \frac{N}{2}.
\end{align*}
In order to evaluate the singular integrals, we used an alternating quadrature rule \cite{Baker:generalized-vortex-methods-free-surface, Sidi-Israeli:quadrature-methods-periodic-singular}:
\begin{align*}
\partial_t z(\alpha_i,t) \approx 2h\frac{\rho^{-} - \rho^{+}}{4\pi} \sum_{j-i \text{ odd}} (z(\al_i,t) - z(\al_j,t))\frac{\sin(z_{1}(\al_i,t) - z_{1}(\al_j,t))}{\cosh(z_{2}(\al_i,t) - z_{2}(\al_j,t)) - \cos(z_{1}(\al_i,t) - z_{1}(\al_j,t))},
\end{align*}
where $h = \frac{2\pi}{N}$ and $N$ is the total number of nodes. We chose the condition at time $t = 0$ to be
\begin{align*}
 z_1(\al,0) = \al - \sin(\al), \qquad z_2(\al,0) = \frac{3 \sin(\al) + 8 \sin(2\al) + 3 \sin(3\al)}4,
\end{align*}
and $N = 2048$ nodes (see Figure \ref{initialconditionTT}).

\begin{figure}[ht]
\centering
\includegraphics[width=0.9\textwidth]{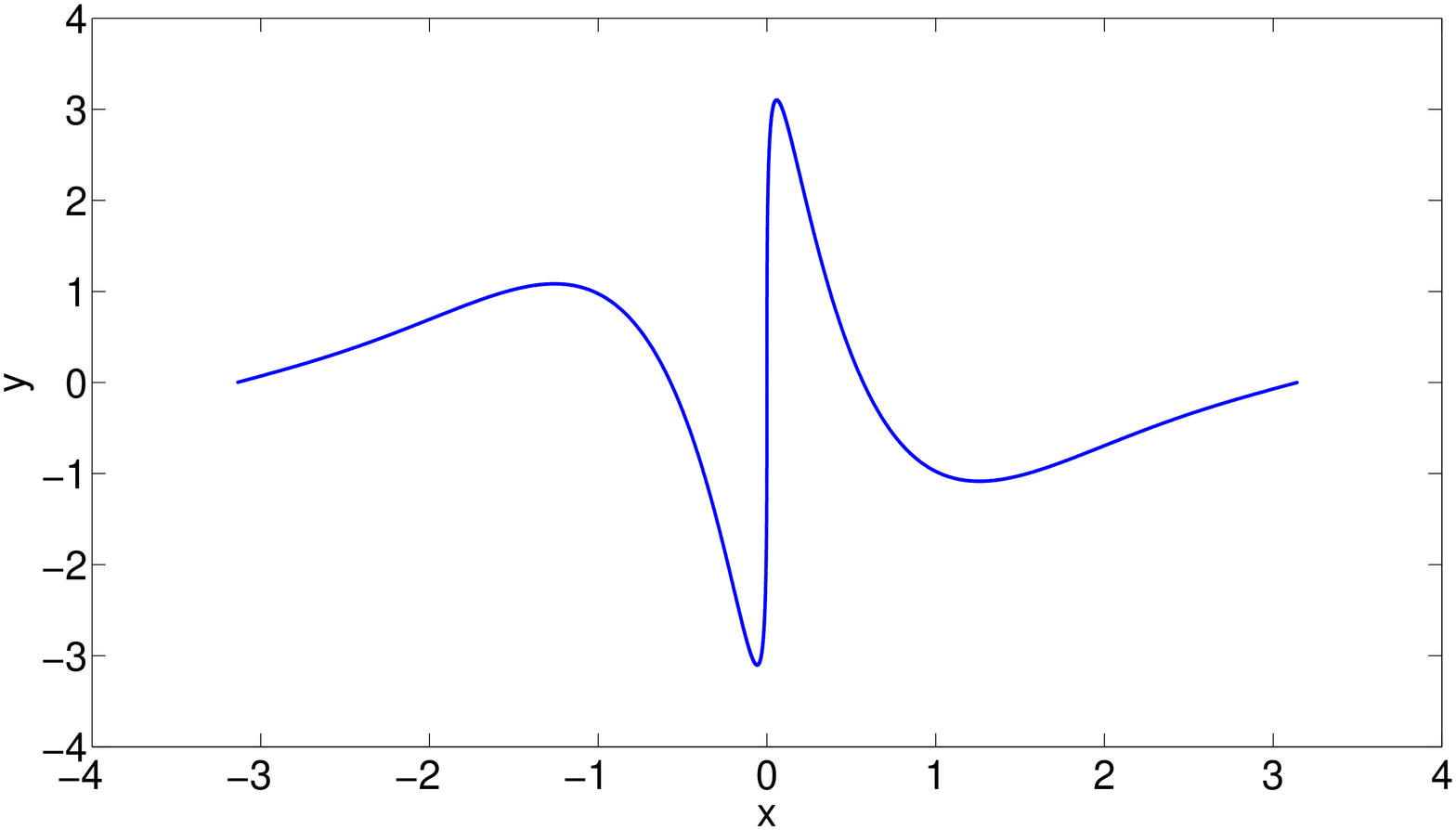}
\caption{$z(\al,0)$ from Section \ref{SectionNumericalResults}.}
\label{initialconditionTT}
\end{figure}

The smoothing was done after every $\Delta t = 4 \cdot 10^{-5}$ and the simulation ran until $t_f = -4.92 \cdot 10^{-2}$. The obtained evolution of $z(\al,t)$ is depicted in Figure \ref{evolutiontt}.

\begin{figure}[ht]
\centering
\includegraphics[width=0.9\textwidth]{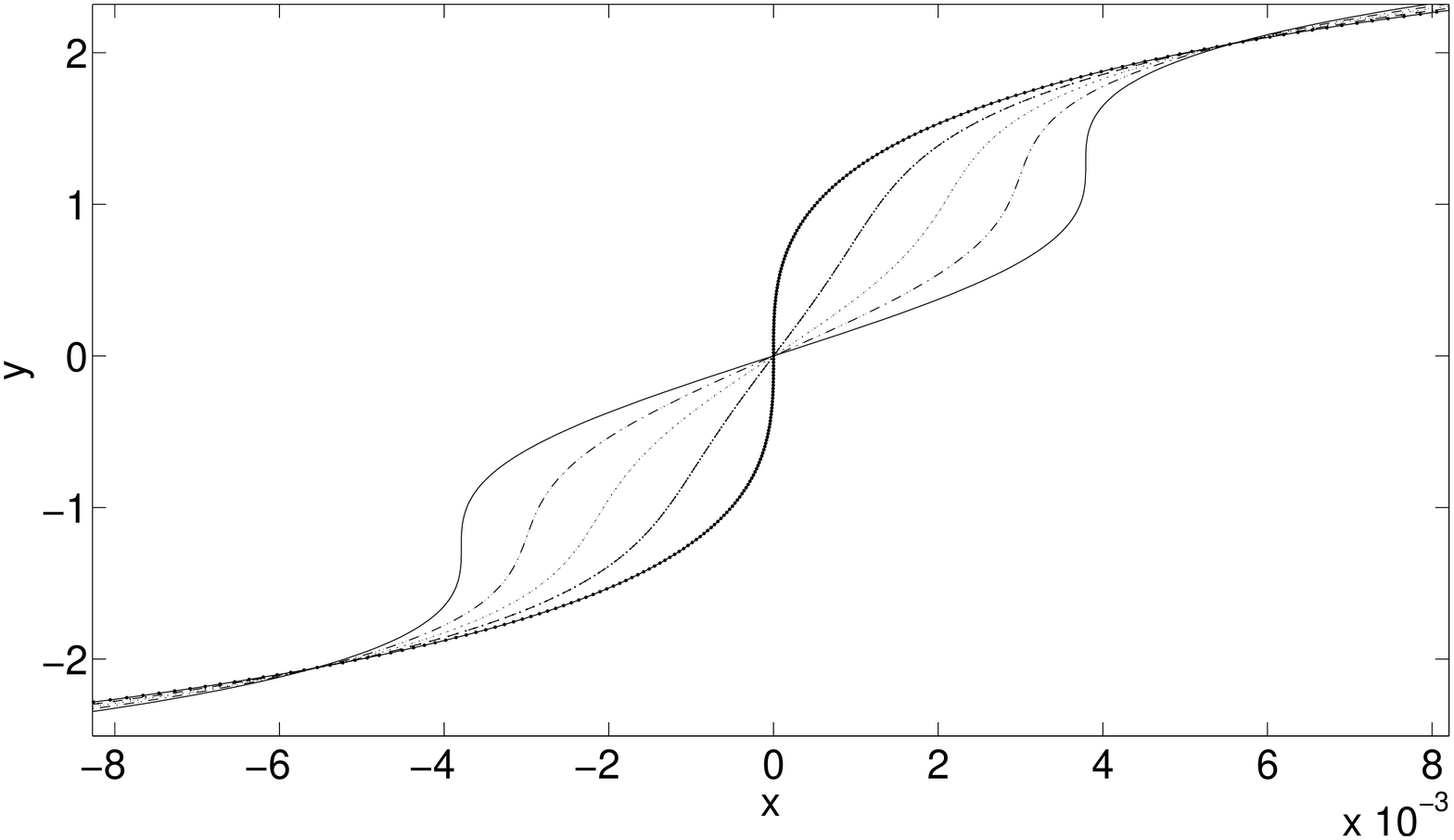}
\caption{Backwards evolution of $z(\al,t)$  from Section \ref{SectionNumericalResults}, with regularization: $t = 0$,  thick line; $t = -1.2 \cdot 10^{-2}$, broken line;  $t = -2.4 \cdot 10^{-2}$, dotted line; $t = -3.6 \cdot 10^{-2}$, broken and dotted line; $t = -4.92 \cdot 10^{-2}$, thin line. Note that this is a closeup of the solution near the origin.}
\label{evolutiontt}
\end{figure}

% Repeticion de la cuenta
Finally, we computed the evolution of equation \eqref{muskatperiodico} forward in time taking as initial condition $z(\al,-4.92 \cdot 10^{-2})$, which has two vertical tangents at  approximately $(\pm 3.795 \cdot 10^{-3}, \pm 1.268)$. The initial data is now in the stable regime and the equation is well-posed. We can use the same integration scheme (both in time and space) with no smoothing. We can see that the curve comes from the unstable regime, then moves to the stable one, and finally returns to the unstable regime.  The whole simulation is carried out in the stable regime.

\section{Statement and proof of the main result}
\label{SectionTheorem}

%Motivated by the above numerics, although the mechanics of the constructed solutions will be somewhat different, we start proving the building blocks of our main theorem:

Our main result, Theorem \ref{theoremshifting} below, is motivated by the above numerics (although the mechanics of the constructed solutions will be somewhat different).  This section is devoted to its proof, which will rest on the following two lemmas.

%Lema 1: todo gira. A partir de aqui fijamos la separacion. Perturbamos analiticamente.

\begin{lemma}
There exists an analytic, odd, asymptotically flat initial condition $z(\al,0)$ whose analytic extension is $H^4$ on the boundary of its strip of analyticity, and also $\al_1>0$ such that the following hold.  We have $\partial_\al z_1(\al,0)
%=\partial^2_\al z_1(\al,0)
=z_2(\al,0)=0< \partial_\al z_2(\al,0)$ for $\al\in\{0,\pm\alpha_1\}$, while $\partial_\al z_1(\al,0)>0$ for all other $\al\in\mathbb{R}$ (in particular, $z(\al,0)$ is a graph of a function of $x$ with three vertical tangents), such that the corresponding solution $z(\al,t)$ of \eqref{muskatinterface} satisfies ${\rm sgn} (t)\partial_\al z_1(\pm\al_1,t)>0$  and ${\rm sgn} (t)\partial_\al z_1(0,t)<0$ for all times $t$ sufficiently close to 0.
%the following for times $t$ sufficiently close to 0.  For $t<0$, a neighborhood of one of the vertical tangents can be parametrized as a graph of a function of $x$ and the other one can not and for positive times their roles interchange: the one that could be parametrized as a graph can not be now and viceversa.
\end{lemma}

By asymptotically flat we mean $z(\al,0)\approx(\al,0)$ and $\partial_\al z(\al,0)\approx (1,0)$ for $|\al|\gg 1$.  Therefore, $z(\al,t)$ will be a graph of a function of $x$ everywhere except near $\al=0$ for small $t>0$, and everywhere except near $\al=\pm\al_1$ for small $t<0$.

\begin{proof}
%[OJO. ESTO YA NO ES SUFICIENTE: HAY QUE MIRAR LA DERIVADA]
%We are going to show that the interface recoils at the vertical tangent points that are different than the origin. To do so, we will calculate $v_1(p+\delta) - v_1(p)$, where $z(p)$ is the vertical tangent point, and show that there exists a $\delta^{*}$ for which for every $\delta < \delta^{*}$ it is strictly positive. In a similar way, we will also show that  $v_1(p-\delta) - v_1(p) < 0$ for every $\delta < \delta^{*}$.
Since in the following we will mostly consider $t=0$, let us denote $z(\al)=z(\al,0)$. 
Following the arguments of \cite{Castro-Cordoba-Fefferman-Gancedo-LopezFernandez:rayleigh-taylor-breakdown}, we find that if  $\al_0$ is any point with $z_1'(\al_0)=z_1''(\al_0)=z_2(\al_0)=0$, then the corresponding velocity $v$ at time $t=0$ satisfies
\begin{align*}
\partial_{\al} v_1(z(\al_0)) = z_2'(\al_0) \int_{\RR} \frac{(z_1(\beta)-z_1(\al_0))z_1'(\beta)z_2(\beta)}{[ (z_1(\al_0) - z_1(\beta))^2+z_2(\beta)^2]^2}d\beta.
\end{align*}
We obviously have the following:
\begin{itemize}
\item If $\partial_{\al} v_1(z(\al_0)) > 0$, then ${\rm sgn}(t) \partial_\al z(\al,t)>0$ for $(\al,t)$ close to $(\al_0,0)$.  In particular, the curve will be in the stable regime near  $\al_0$ for small $t>0$.
\item If $\partial_{\al} v_1(z(\al_0)) < 0$, then ${\rm sgn}(t) \partial_\al z(\al,t)<0$ for $(\al,t)$ close to $(\al_0,0)$.  In particular, the curve will be in the unstable regime (i.e., it will turn over)  near  $\al_0$ for small $t>0$.
\end{itemize}

Our $z(\al)$ will consist of three building blocks: center and two tails.  First let
\begin{align*}
 z^{t}_1(\al) = 
\left\{
\begin{array}{cc}
\al & -2 \leq \al \leq -1 \\
\al^3 & -1 \leq \al \leq 1 \\
\al & 1 \leq \al \leq 2 \\
\end{array}
\right.
\quad
z^{t}_2(\al) = 
\left\{
\begin{array}{cc}
-\al-2 & -2 \leq \al \leq -1 \\
 \al & -1 \leq \al \leq 1 \\
-\al + 2 & 1 \leq \al \leq 2 \\
\end{array}
\right.
\end{align*}
and 
\begin{align*}
 z_1^{c}(\al) = 
\left\{
\begin{array}{cc}
\al & -7 \leq \al \leq -1 \\
\al^3 & -1 \leq \al \leq 1 \\
\al & 1 \leq \al \leq 7 \\
\end{array}
\right.
\quad
z_2^{c}(\al) = 
\left\{
\begin{array}{cc}
\frac{3}{2}\al + \frac{21}{2} & -7 \leq \al \leq -5 \\
3 & -5 \leq \al \leq -2 \\
-5\al - 7 & -2 \leq \al \leq -1 \\
 3\al - \al^{3} & -1 \leq \al \leq 1 \\
-5\al + 7 & 1 \leq \al \leq 2 \\
-3 & 2 \leq \al \leq 5 \\
\frac{3}{2}\al - \frac{21}{2} & 5 \leq \al \leq 7 \\
\end{array}
\right.
\end{align*}
(note that both are odd).  
We will now consider curves $z^{R}(\al)$ of the following form:
\begin{align*}
 z_1^{R}(\al) = 
\left\{
\begin{array}{cc}
z_1^{t}(\al+R)-R & -R-2 \leq \al \leq -R+2 \\
z_1^{c}(\al) & -7 \leq \al \leq 7 \\
z_1^{t}(\al-R)+R & R-2 \leq \al \leq R+2 \\
\al & \text{otherwise}
\end{array}
\right.
\quad
 z_2^{R}(\al) = 
\left\{
\begin{array}{cc}
z_2^{t}(\al+R) & -R-2 \leq \al \leq -R+2 \\
z_2^{c}(\al) & -7 \leq \al \leq 7 \\
z_2^{t}(\al-R) & R-2 \leq \al \leq R+2 \\
0 & \text{otherwise}
\end{array}
\right.
\end{align*}
 where $R>9$ will be fixed later.  A sketch of the initial condition $z^R(\al)$, which satisfies all hypotheses of the lemma (with $\al_1=R$) except regularity, appears in Figure \ref{initialconditionshift}.  We will now show that for small $t>0$, the corresponding solution of \eqref{muskatinterface} will turn over near $\al=0$ but not near $\al=\pm R$.  To do so, 
 %we need to compute the different contributions to $\partial_{\al} v_1(z(\al_0))$. We 
 we will split the integrals in the formula for $\partial_{\al} v_1(z^R(\al_0))$ ($\al_0=0,\pm R$) into  center-center, tail-tail, center-tail, and tail-center contributions (note also that $(z^R_2)'(\al_0)>0$).  We will compute the first two and only estimate the last two since they will be made arbitrarily small by choosing $R$ large enough.

\begin{figure}[ht]
\centering
\includegraphics[width=0.9\textwidth]{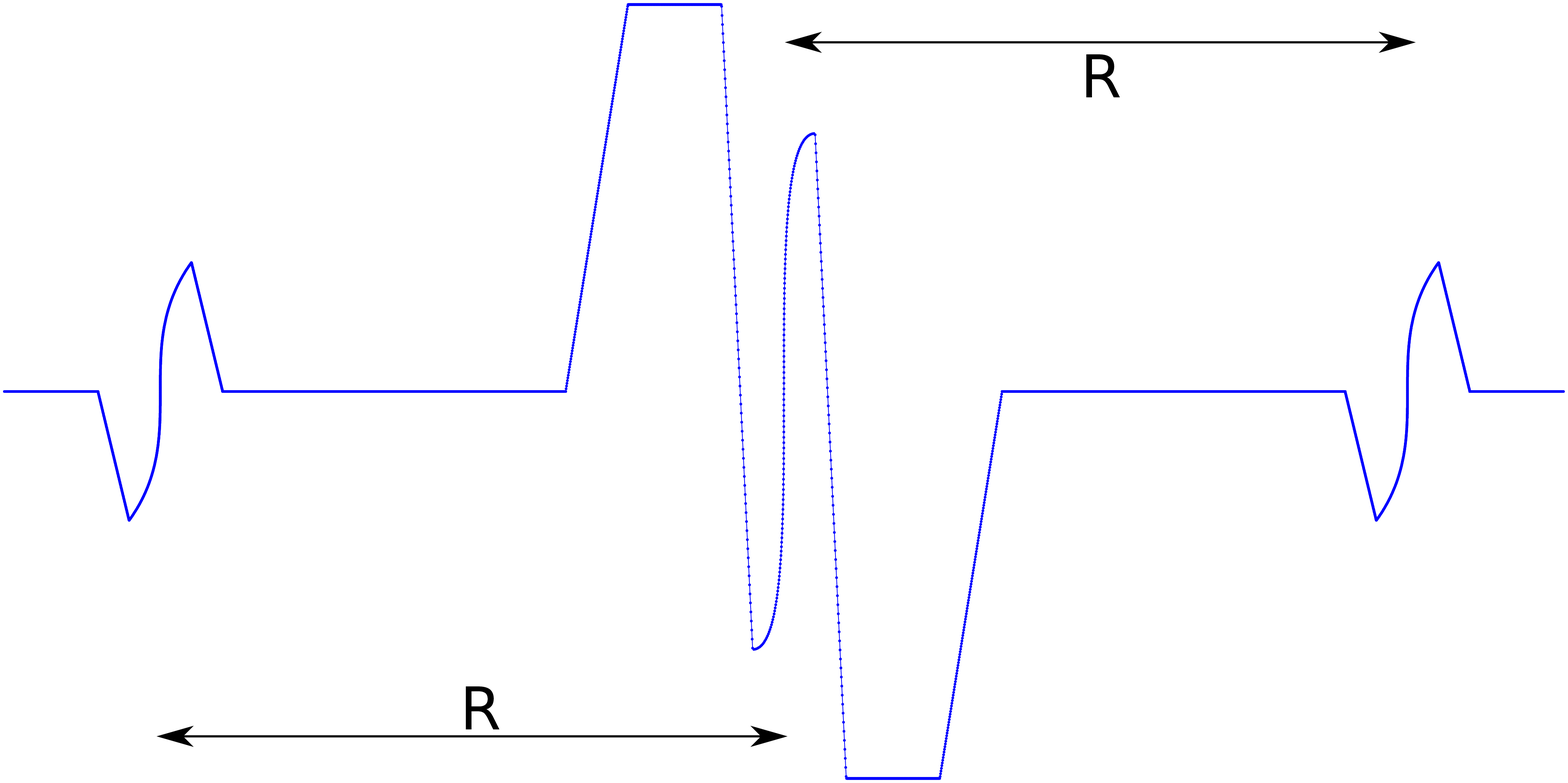}
\caption{$z^{R}(\al)$ from Section \ref{SectionTheorem}.}
\label{initialconditionshift}
\end{figure}

We start with center-center. Since $z^c$ is odd, the quantity we are interested in simplifies to
\begin{align*}
I_{cc}=2 \int_{0}^{\infty} \frac{z^c_1(\beta)(z^c_1)'(\beta)z^c_2(\beta)}{(z^c_1(\beta)^2+z^c_2(\beta)^2)^2}d\beta  = I_{cc}^{1} + I_{cc}^{2} + I_{cc}^{3} + I_{cc}^{4},
\end{align*}
%
%
%We can compute the expression of $I_{cc}$ analytically by splitting into
%
%
%\begin{align*}
%I_{cc} = I_{cc}^{1} + I_{cc}^{2} + I_{cc}^{3} + I_{cc}^{4},
%\end{align*}
%
where
\begin{align*}
I_{cc}^{1} & = -2\int_{0}^{1} \frac{3x^2(x^2-3)}{(2x^4-6x^2+9)^2} dx \\
I_{cc}^{2} & = 2\int_{1}^{2} \frac{7 x - 5 x^2}{(26x^2 - 70 x + 49)^{2}} dx \\
I_{cc}^{3} & = -2\int_{2}^{5} \frac{3x}{(9+x^2)^2} dx \\
I_{cc}^{4} & = \int_{5}^{7} \frac{3x^2-21x}{\left(\frac{441}{4} - \frac{63x}{2}+  \frac{13x^2}{4}\right)^2} dx.
\end{align*}
%Note that we have used the fact that $z$ is odd to only evaluate the integral between 0 and infinity.
The last three integrals can be explicitly calculated:
\begin{align*}
I_{cc}^{2} & = \left.\frac{1}{26}\left(\frac{-7 + 10 x}{ 26x^2 - 70x + 49}\right)\right|_{x=1}^{x=2}
= \frac{1}{65} \\
I_{cc}^{3} & = \left.3\frac{1}{9+x^2}\right|_{x=2}^{x=5} = -\frac{63}{442} \\
I_{cc}^{4} & = \left.\frac{48 (7 - 2 x)}{338 x^2 - 3276x + 11466}\right|_{x=5}^{x=7} = -\frac{3}{119}.
\end{align*}
The first one can also be calculated explicitly, with slightly more effort. For the sake of simplicity, we do not present here the full symbolic integral, only an approximation of the final result:
\begin{align*}
I_{cc}^{1} & \approx 0.127271158.
\end{align*}
%-(1/96)*sqrt(-3+3*sqrt(2))*ln(2+3*sqrt(2)+2*sqrt(2)*sqrt(-3+3*sqrt(2))+2*sqrt(-3+3*sqrt(2)))+(1/96)*sqrt(-3+3*sqrt(2))*ln(2+3*sqrt(2)-2*sqrt(2)*sqrt(-3+3*sqrt(2))-2*sqrt(-3+3*sqrt(2)))-(1/48)*sqrt(-3+3*sqrt(2))*arctan(-(2/3)*sqrt(2)*sqrt(-3+3*sqrt(2))-(2/3)*sqrt(-3+3*sqrt(2))+sqrt(2)+1)*sqrt(2)-(1/48)*sqrt(-3+3*sqrt(2))*arctan(-(2/3)*sqrt(2)*sqrt(-3+3*sqrt(2))-(2/3)*sqrt(-3+3*sqrt(2))+sqrt(2)+1)+(1/48)*sqrt(-3+3*sqrt(2))*arctan((2/3)*sqrt(2)*sqrt(-3+3*sqrt(2))+(2/3)*sqrt(-3+3*sqrt(2))+sqrt(2)+1)*sqrt(2)+(1/48)*sqrt(-3+3*sqrt(2))*arctan((2/3)*sqrt(2)*sqrt(-3+3*sqrt(2))+(2/3)*sqrt(-3+3*sqrt(2))+sqrt(2)+1)+1/20
Adding all the contributions, we obtain
\begin{align*}
I_{cc} \approx -0.0250882 < 0.
\end{align*}

%Contribucion de fuera

Next, the tail-center term (i.e., the contribution of the tails to $\partial_{\al} v_1(z^R(0))$) is
\begin{align*}
I_{tc} = 2\int_{-2}^{2} \frac{(R+z_1^{t}(\beta)) (z_1^{t})'(\beta)z_2^{t}(\beta)}{[ (R+z_1^{t}(\beta))^2+z_2^{t}(\beta)^2]^2}d\beta.
\end{align*}
The explicit expression for $z^t$ yields the easy bound
\begin{align*}
 |I_{tc}| \leq 2\frac{(R+2)\cdot 3 \cdot 1}{(R-2)^4} \cdot 4 %\leq \frac{40}{(R-2)^3}
\end{align*}
Hence, by choosing $R$ large enough we can ensure that
\begin{align}
 \label{conditionRcenter}
 I_{cc} + I_{tc} < 0.
\end{align}

%Extremos
Let us now consider the tail-tail contribution. Because of symmetry  it suffices to consider $\al_0 = R$.  Then we have
\begin{align*}
 I_{tt} = I_{tt}^{1} + I_{tt}^{2} + I_{tt}^{3},
\end{align*}
where (after using the explicit expression for $z^t$)
\begin{align*}
I_{tt}^{1} & = -\int_{-2}^{-1}  \frac{x (2 + x)}{(2x^2 + 4x + 4)^2} dx \\
I_{tt}^{2} & = \int_{-1}^{1} \frac{3x^2}{(1+x^4)^2} dx \\
I_{tt}^{3} & = -\int_{1}^{2} \frac{x(-2+x)}{(2x^2 -4x + 4)^2} dx. \\
\end{align*}
The first  and last integral can again easily be evaluated explicitly:
\begin{align*}
 I_{tt}^{1} & = \left.\frac{1 + x}{4 (2 + 2 x + x^2)}\right|_{x=-2}^{x=-1} = \frac{1}{8} \\
I_{tt}^{3} & = \left.\frac{x-1}{4 (2 - 2 x + x^2)}\right|_{x=1}^{x=2} = \frac{1}{8},
\end{align*}
whereas $I_{tt}^{2} > 0$ since its integrand is positive. We can thus conclude that $I_{tt} > \frac{1}{4}$.

%Contribucion de fuera

Finally, we bound the center-tail contribution
\begin{align*}
 I_{ct} = I_{ct}^{1} + I_{ct}^{2},
\end{align*}
where
\begin{align*}
 I_{ct}^{1} & = \int_{-7}^{7} \frac{(z^c_1(\beta)-R) (z^c_1)'(\beta)z^c_2(\beta)}{[ (R-z^c_1(\beta))^2+z^c_2(\beta)^2]^2}d\beta \\
 I_{ct}^{2} & = \int_{-R-2}^{-R+2} \frac{(z^t_1(\beta)-R) (z^t_1)'(\beta)z^t_2(\beta)}{[ (R-z^t_1(\beta))^2+z^t_2(\beta)^2]^2}d\beta.
\end{align*}
We can easily obtain the bounds
\begin{align*}
 |I_{ct}^{1}| \leq \frac{(R+7)\cdot 3 \cdot 3 \cdot 14}{(R-7)^4} \\
|I_{ct}^{2}| \leq \frac{(2R+2) \cdot 3 \cdot 1 \cdot 4}{(2R-2)^{4}},
\end{align*}
so again, by choosing $R$ large enough we can ensure that
\begin{align}
 \label{conditionRtail}
 I_{tt} + I_{ct} > 0.
\end{align}

We therefore choose $R$ such that conditions \eqref{conditionRcenter} and \eqref{conditionRtail} are satisfied. 
%Perturbacion analitica
%[Esto esta mal. No puede haber una funcion analitica con pedazos de soporte compacto. Pero se puede arreglar con el argumento que da Zlatos.]
%
Then we let $z^{R}_{ana}(\al)$ be a sufficiently close analytic
%, $H^4$ on the boundary of its strip of analyticity 
perturbation of $z^{R}(\al)$ which satisfies all the hypotheses of the lemma with $\al_1=R$ (so it also has vertical tangents at $\alpha = 0, \pm R$).  It is not difficult to see that this can be done so that $\partial_{\al} v_1(z^R_{ana}(\al_0))$ has the same sign as $\partial_{\al} v_1(z^R(\al_0))$ for $\al_0=0,\pm R$.  Then $z(\al,0)=z^{R}_{ana}(\al)$ will be the desired initial condition.
\end{proof}

\begin{lemma}\label{lemaestabilidad}
 Let $z(\al,t)$ and $w(\al,t)$ be two analytic solutions of equation \eqref{muskatinterface} with initial conditions $z_0(\al)$ and $w_0(\al)$ respectively. Let $d(\al,t) = z(\al,t) - w(\al,t)$, and for $\al,\beta\in \RR$ let 
\begin{align*}
F(z,t)(\al,\beta)=\frac{\beta^2}{|z(\al,t)-z(\al-\beta,t)|^2}.
\end{align*}
Assume $F(z,0), F(w,0) \in L^{\infty}$ and consider the energy
\begin{align*}
%E(t) = \sum_{\pm}\int(|z(\al\pm ict)|^2+|\da^4z(\al\pm ict)|^2) d\al,
E(t) = \|d(\cdot,t)\|^{2}_{H^{4}(S(t))}  = \sup_{|c| < \xi(t)} \int_{\mathbb{R}} \left(|d(\al+ ic,t)|^2+|\da^4d(\al+ ic,t)|^2 \right) d\al,
%= \sum_{\pm}\int_{\mathbb{R}} \left(|d(\al\pm i\xi(t),t)|^2+|\da^4d(\al\pm i\xi(t),t)|^2 \right) d\al,
\end{align*}
with $S(t)=\{\al + i c : |c| < \xi(t)\}$ the strip of analyticity of the function $d(\cdot,t)$.
Then there exists a polynomial $P(x,y,q,r,s)$ such that
\begin{align*}
 \frac{d}{dt}E(t) \leq P \left(E(t),\|z_0\|_{H^{4}(S(0))},\|w_0\|_{H^{4}(S(0))},\|F(z,0)\|_{L^{\infty}(S(0))},\|F(w,0)\|_{L^{\infty}(S(0))} \right),
\end{align*}
where the $L^{\infty}$ norm in the strip is defined as
\begin{align*}
 \|F(z,t)\|_{L^{\infty}(S(t))} = \sup_{|c| < \xi(t)}\|F(z,t)(\alpha+ic,\beta)\|_{L^{\infty}(\alpha,\beta)}.
\end{align*}

\end{lemma}

\begin{proof}
The proof appears in \cite[Section 6, p. 940, eq. (44)]{Castro-Cordoba-Fefferman-Gancedo-LopezFernandez:rayleigh-taylor-breakdown}. One only has to write the equation that $d(\al,t)$ satisfies, then apply the same estimates as in the local existence theorem from \cite{Castro-Cordoba-Fefferman-Gancedo-LopezFernandez:rayleigh-taylor-breakdown}, and finally control the evolution of the norms of $z$ and $w$ in terms of the norms of $z_0$ and $w_0$ via their respective local existence theorems.
\end{proof}

%[Cambiar y usar el argumento de Zlatos aqui]

%Finally, let $t_1(\eps_1)$ be the maximum time that we can guarantee through Lemma \ref{lemaestabilidad} that the rightmost curve will be a graph if we consider perturbations of the following type: construct $z_{ana}^{R,\eps_1}(\al)$ gluing together $z^{t}_{ana}(\al-R)$ in the tails with $z^{c}(\al,-\eps)$ in the center, where $z^{c}(\al,-\eps_1)$ is the solution at time $t = -\eps$ of \eqref{muskatinterface} with initial condition $z^{c}_{ana}(\al)$. Those curves will also turn after a time close to $\eps$.

%By Lemma \ref{lemaestabilidad}, we know that $t_1$ exists if $\eps_1$ is small enough. We also know that $t_1(0) > 0$ and $t_1$ is continuous. Then, there exists an $\eps^{*} > 0$ such that $t_1(\eps^{*}) > \eps^{*}$. This concludes the proof of the Theorem.
%Finally, by taking $\eps = \min\{\eps^{*}, t_1(\eps^{*}) - \eps^{*}\}$, this concludes the proof of the Theorem.

We are now ready to prove our main result.

\begin{theorem}\label{theoremshifting}
 There exists an analytic initial curve $z(\al,0)$ whose analytic extension is $H^4$ on the boundary of its strip of analyticity, and also times $-T<t_1<0<t_0<T$ and $\varepsilon>0$ such that the following hold.  The corresponding solution of \eqref{muskatinterface}  exists for $t\in (-T,T)$ in the class of analytic functions of $\al$ whose analytic extensions are $H^4$ on the boundaries of their strips of analyticity, and it is a graph of a function of $x$ for each $t\in(t_1,t_0)$ (i.e., it is in the stable regime for these $t$) but not for $t\in(t_1-\varepsilon,t_1)\cup (t_0,t_0+\varepsilon)$ (i.e., it is in the unstable regime for these $t$). The solution develops vertical tangents at times $t_1$ and $t_0$. 
\end{theorem}

\begin{proof}
For $\delta\ge 0$, let $z^{\delta}(\al)$ be a sufficiently close perturbation of $z^{R}_{ana}(\al)$ with the same properties except that its tangents at $\al=0,\pm R$ are $\delta$ away from vertical (in the sense of $\partial_{\al} z_1^\delta(\al_0)=\delta$ for $\al_0=0,\pm R$, while $\partial_{\al} z_2^\delta(\al_0)$ remains away from 0, uniformly in $\delta$).  It is not difficult to see that this can be chosen with a radius of analyticity away from 0, uniformly in $\delta$.  Then the solutions $z^\delta(\al,t)$ of \eqref{muskatinterface}  corresponding to initial conditions $z^\delta(\al)$ exist in the class of analytic functions on the interval $(-T,T)$ for some $\delta$-independent $T>0$.

Let $t_0(\delta)$ be the time of turnover of $z^\delta(\al,t)$ near $\al_0=0$, and let $t_1(\delta)$ be the time of turnover near $\al_0=\pm R$ (these exist if $z^{\delta}(\al)$ is close enough to $z^{R}_{ana}(\al)$). We have $t_0(0)=t_1(0) = 0$, as well as  $t_0(\delta) > 0>t_1(\delta)$ for sufficiently small $\delta$ and $\lim_{\delta\to 0} t_0(\delta)=\lim_{\delta\to 0} t_1(\delta)=0$, due to Lemma \ref{lemaestabilidad}.  Choosing $\delta$ such that $-t_1(\delta),t_0(\delta)<T$ yields the desired initial condition $z(\al,0)=z^\delta(\al)$.
%Let $t_2(\delta)$ be the time we are guaranteed by the stability theorem that the solution will stay in the stable regime. We have that $t_2(0) = T > 0$. We can then pick $\eps^{*} > 0$ satisfying $0 < t_1(\eps^{*}) < t_2(\eps^{*})$. Let $t_3$ be the minimum time of existence of $z^{\delta}(\al), \delta \in [0,\eps^{*}]$. Finally, we can take as initial data $z(\al) = z^{\eps}(\al)$, where $\eps = \min\{\eps^{*},t_3\}$. This concludes the proof.
\end{proof}

\section{Further discussion}
\label{SectionDiscussion}

We now discuss a couple of remarks related to Theorem \ref{theoremshifting}. We plan to address these in the future.

The first is the question of narrowing the gap between Theorems \ref{oneone} and \ref{weakSOLthm}, which was our initial motivation.  Based on our numerical simulations,  we propose the following conjecture:

\begin{conjecture}
There exists a solution $f(x,t)$ of equation \eqref{muskatinterfacegraph} that has $\|f(\cdot,0)\|_{L^{\infty}} = 50$ and turns over in finite time. 
\end{conjecture}

We present in Figure \ref{50Derivative} the supporting numerical evidence, with initial condition
%, whose graph is $f(\cdot,0)$, was taken to be
\begin{align*}
 z_{1}(\al,0) = \al - 0.96 \sin(\al), \quad z_{2}(\al,0) = \frac{2}{3}\sin(3\al).
\end{align*}
If we  reparametrize this curve as $(x,f(x))$, then $\|f'\|_{L^\infty}=f'(0)=50$.

\begin{figure}[ht]
\centering
\includegraphics[width=0.9\textwidth]{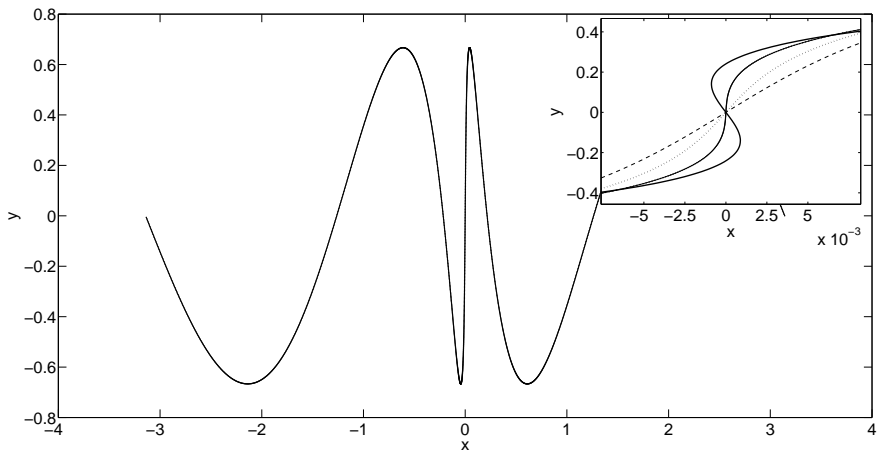}
\caption{$f(x,0)$ from Section \ref{SectionDiscussion}. Inset contains zoom (around $x = 0$) of $f(x,t)$ for different times: $t = 0$, broken line; $t = 5.86 \cdot 10^{-2}$, dotted line; $t = 1.26 \cdot 10^{-1}$, thin line; $t = 1.93 \cdot 10^{-1}$, thick line.}
\label{50Derivative}
\end{figure}

A possible strategy to proving this conjecture, outlined in \cite{Castro-Cordoba-Fefferman-Gancedo-GomezSerrano:stability-splash-singularities-water-waves}, is as follows.  One can consider an approximate solution $w(\al,t)$, satisfying (with a small error $err(\al,t)$)
\begin{align}\label{muskatinterfacenumerics}
\partial_{t}w(\al,t) = \frac{\rho^{-} - \rho^{+}}{2\pi} \int \frac{w_1(\al,t) - w_1(\beta,t)}{|w(\al,t) - w(\beta,t)|^{2}}(\partial_{\al}w(\al,t) - \partial_{\beta}w(\beta,t)) d\beta + err(\al,t)
\end{align}
Next, subtracting equation \eqref{muskatinterfacenumerics} from \eqref{muskatinterface} and defining $d(\al,t) = z(\al,t) - w(\al,t)$, one can write a system of equations for $d(\al,t)$ in such a way that only $d(\al,t)$ $w(\al,t)$ and $err(\al,t)$ appear, and $z(\al,t)$ does not. We recall that $w(\al,t)$ is explicit since it is the numerically calculated function and 
%$\|\partial_{\al} y(\cdot,0)\|_{L^{\infty}} = 50$ and 
$\partial_{\al} w_1(\al,T) < 0$ for some explicit $(\al,T)$. In a similar fashion as in the local existence theorem from \cite{Castro-Cordoba-Fefferman-Gancedo-LopezFernandez:rayleigh-taylor-breakdown}, a stability theorem follows, and this gives explicit bounds on the evolution of some norms of $d(\al,t)$, in particular it controls the $L^{\infty}$ norm of $\partial_{\al} d_{1}(\al,t)$. The bounding of the constants of the stability theorem can be done either via traditional pencil-and-paper means, or using rigorous computer-assisted bounds and interval arithmetics, or a combination of both. Finally, if $\|\partial_{\al} d_1(\al,T)\|_{L^{\infty}}$ is small enough, then $\partial_{\al} z_1(\al,T) = \partial_{\al} w_1(\al,T) + \partial_{\al} d_1(\al,T) < 0$. Note that, as opposed to Lemma \ref{lemaestabilidad}, we need to have good enough bounds on the constants that appear in the inequality, not just to know their existence. This makes the problem considerably harder.

\begin{rem}
We believe that the constant 50 is not sharp and  can be improved with further numerical search and better estimates.
\end{rem}

\begin{rem}
Showing existence of a stability shift in the other direction (stable $\rightarrow$ unstable $\rightarrow$ stable) seems harder, even though we have numerics that exhibit that behaviour. One has to produce an initial condition that first turns over and then recoils back, so in contrast with Theorem \ref{theoremshifting}, the solution lives in the unstable regime during the ``middle'' time interval of its evolution.

A similar strategy as outlined above could work. After finding a numerical guess for a such solution, one can try to  show via a stability theorem and computer-assisted estimates that close to this guess there exists a true solution of \eqref{muskatinterface} which still exhibits this phenomenon.
\end{rem}

\section*{Acknowledgements}

DC and JGS were partially supported by the grant MTM2011-26696 (Spain) and ICMAT Severo Ochoa project SEV-2011-0087. AZ acknowledges partial support by NSF grants DMS-1056327 and DMS-1159133.

 \bibliographystyle{abbrv}
 \bibliography{references}

\begin{tabular}{l}
\textbf{Diego C\'ordoba} \\
  {\small Instituto de Ciencias Matem\'aticas} \\
 {\small Consejo Superior de Investigaciones Cient\'ificas} \\
 {\small C/ Nicolas Cabrera, 13-15, 28049 Madrid, Spain} \\
  {\small Email: dcg@icmat.es} \\
\\
 {\small Department of Mathematics} \\
 {\small Princeton University} \\
 {\small 804 Fine Hall, Washington Rd,} \\
  {\small Princeton, NJ 08544, USA} \\
 {\small Email: dcg@math.princeton.edu} \\
\\
\textbf{Javier G\'omez-Serrano} \\
{\small Department of Mathematics} \\
{\small Princeton University}\\
{\small 610 Fine Hall, Washington Rd,}\\
{\small Princeton, NJ 08544, USA}\\
 {\small Email: jg27@math.princeton.edu} \\
  \\
\textbf{Andrej Zlato\v{s}} \\
{\small Department of Mathematics}\\
{\small University of Wisconsin}\\
{\small 480 Lincoln Dr}\\
{\small Madison, WI 53706, USA}\\
{\small Email: andrej@math.wisc.edu}\\
    \\

\end{tabular}

\end{document}